\documentclass[11pt]{article}

\usepackage[a4paper,margin=29mm]{geometry}
\usepackage[T1]{fontenc}
\usepackage{lmodern}
\usepackage{amsmath,amssymb,amsthm}
\allowdisplaybreaks
\usepackage{booktabs}
\usepackage{microtype}
\usepackage[hidelinks]{hyperref}
\hypersetup{
  pdftitle={Quaternary Legendre Pairs of Lengths 42, 46, 52, 58, 66, 72, and 80},
  pdfauthor={Nikita Lebedev},
  pdfsubject={Combinatorial design theory; quaternary Legendre pairs; Hadamard matrices},
  pdfkeywords={quaternary Legendre pair, Hadamard matrix, periodic autocorrelation,
    supplementary difference family, exhaustive search}
}

\newtheorem{theorem}{Theorem}
\newtheorem{lemma}[theorem]{Lemma}
\newtheorem{proposition}[theorem]{Proposition}
\newcommand{\PAF}{\operatorname{PAF}}
\newcommand{\PSD}{\operatorname{PSD}}
\newcommand{\DFT}{\operatorname{DFT}}
\newcommand{\ii}{\mathrm{i}}
\newcommand{\roots}{\{\pm 1,\pm\ii\}}

\title{Quaternary Legendre Pairs of Lengths\\42, 46, 52, 58, 66, 72, and 80}
\author{Nikita Lebedev}
\date{4 August 2026}

\begin{document}
\maketitle

\begin{abstract}
A quaternary Legendre pair of length $\ell$ is a pair of sequences over the
fourth roots of unity whose periodic autocorrelations sum to $-2$ at every
nonzero shift.  Each such pair yields a quaternary Hadamard matrix of order
$2\ell+2$ and a binary Hadamard matrix of order $4\ell+4$.  Fourteen even
lengths up to 100 remained open in the 2025 tables of Kotsireas, Koutschan and
Winterhof and of Jedwab and Pender.  We settle seven of them --- the four
smallest open lengths, 42, 46, 52, and 58, together with 66, 72, and 80 ---
and print every pair in full.  Checking one takes
a few lines of code and exact arithmetic over the Gaussian integers, so the
main results can be confirmed without running any of our software.

Each search writes a pair as four binary sequences under a Gray map,
restricts to reflection-symmetric rows, imposes an exact spectral constraint
at the midpoint frequency, and matches candidates meet-in-the-middle on their
autocorrelation vectors; from length 52 onward, exact compression along
divisor chains shrinks the search before any candidate is lifted.  Everything
is deterministic: no seed, heuristic stopping rule, or floating-point value
enters an acceptance decision.  At lengths 64, 70 and 76 the same machinery ran to exhaustion and
found nothing: the reversible four-block queue built at each length contains
no pair.  These are exhaustions of a specific recorded queue, not nonexistence
results, and the certificates do not by themselves establish that the queue
captures every reversible object of its length.  A closing section sorts every
claim in the paper by what a reader
must accept in order to believe it, and states exactly how far the
accompanying artifacts support each one.
\end{abstract}


\section{Introduction}

A quaternary Legendre pair (qLP) of even length $\ell$ is a pair
$A,B\in\roots^\ell$ satisfying
\begin{equation}
  \PAF(A,k)+\PAF(B,k)=-2
  \qquad (1\leq k<\ell),
  \label{eq:qlp}
\end{equation}
where indices are modulo $\ell$ and
$\PAF(X,k)=\sum_{j=0}^{\ell-1}X_j\overline{X_{j+k}}$.
Kotsireas and Winterhof introduced quaternary Legendre pairs --- generalizing
the binary Legendre pairs of Fletcher, Gysin and Seberry
\cite{fletcher-gysin-seberry-2001} --- and showed that
they yield quaternary Hadamard matrices of order $2\ell+2$ and binary
Hadamard matrices of order $4\ell+4$ \cite{kotsireas-winterhof-2024}.
The 2025 published lists
\cite{kotsireas-koutschan-winterhof-2025,jedwab-pender-2025} left the
following lengths at most 100 unresolved:
\[
  42,46,52,58,64,66,70,72,76,80,88,92,94,100.
\]

This paper resolves seven of them --- the four smallest, 42, 46, 52, and 58,
together with 66, 72, and 80.  Because
the resulting objects are short literal strings and the defining equations use
only exact integer arithmetic, the main results are cheap to verify and do not
depend on trusting our software; Section~\ref{sec:pairs} states them first,
before any machinery.

At three lengths that resisted the same searches --- 64, 70 and 76 --- we
contribute an exact negative instead: the reversible four-block queue
constructed at each was searched to exhaustion and contains no pair.  These
are statements about a specific recorded queue, not nonexistence statements,
and
Section~\ref{sec:terminals} is explicit about the difference and about the
limits of the accompanying evidence.

The remainder is organized so that the strength of each claim is visible.
Section~\ref{sec:pairs} gives the pairs.  Sections~\ref{sec:reduction}
and~\ref{sec:searches} describe the reduction and the searches that found
them.  Section~\ref{sec:terminals} gives the negative results,
Section~\ref{sec:structure} the structural context, and
Section~\ref{sec:known} some explicit pairs at lengths already resolved in
the literature, which serve here mainly as validation of the same pipeline.
Section~\ref{sec:verification} classifies every claim in the paper by what a
reader must accept in order to believe it, and the four levels are worth
stating now.  The pairs require nothing: decode the printed strings and check
the defining identity.  The lemmas and propositions require reading a proof.
The compressed censuses are backed by counts a reader can derive independently
of our code.  Of the headline claims, only the three exhaustions, at lengths
64, 70 and 76,
require either trusting our implementations or repeating the computation --- and
Section~\ref{sec:verification} also lists, explicitly, the assurances the
accompanying material does not provide, and distinguishes the three lengths,
which are not equally well evidenced.  The software behind all of it was written
by AI systems under the author's direction; Section~\ref{sec:provenance} states
that plainly and explains why it is the reason the paper is arranged as it is.

\section{The pairs}
\label{sec:pairs}

Encode $1,\ii,-1,-\ii$ by the digits $0,1,2,3$ respectively.

\begin{theorem}
\label{thm:pairs}
Quaternary Legendre pairs exist at lengths $42$, $46$, $52$, $58$, $66$,
$72$, and $80$.  Explicit examples are the following.
\end{theorem}

\begin{proof}
Take the pairs
\begin{align*}
A_{42}&=\texttt{0200301203231110313222223130111323021030}\\*[-2pt]
      &\phantom{{}={}}\texttt{02},\\
B_{42}&=\texttt{0233320121011303110321230113031101210233}\\*[-2pt]
      &\phantom{{}={}}\texttt{32},\\
A_{46}&=\texttt{3202200013023311120321313123021113320310}\\*[-2pt]
      &\phantom{{}={}}\texttt{002202},\\
B_{46}&=\texttt{0222111332130001013023313320310100031233}\\*[-2pt]
      &\phantom{{}={}}\texttt{111222},\\
A_{52}&=\texttt{0012122022312330100232030101030232001033}\\*[-2pt]
      &\phantom{{}={}}\texttt{213220221210},\\
B_{52}&=\texttt{3201003311331202313322101101101223313202}\\*[-2pt]
      &\phantom{{}={}}\texttt{133113300102},\\
A_{58}&=\texttt{3322100120221122023313002030010030200313}\\*[-2pt]
      &\phantom{{}={}}\texttt{320221122021001223},\\
B_{58}&=\texttt{0212103032123022301010332013111310233010}\\*[-2pt]
      &\phantom{{}={}}\texttt{103220321230301212},\\
A_{66}&=\texttt{2213010031221002233133331030111220221110}\\*[-2pt]
      &\phantom{{}={}}\texttt{30133331332200122130010312},\\
B_{66}&=\texttt{1030320011031331120021213232202020202022}\\*[-2pt]
      &\phantom{{}={}}\texttt{32312120021133130110023030},\\
A_{72}&=\texttt{2301312223030012312201110313100320232320}\\*[-2pt]
      &\phantom{{}={}}\texttt{23001313011102213210030322213103},\\
B_{72}&=\texttt{1003012103130023332202012301112310222220}\\*[-2pt]
      &\phantom{{}={}}\texttt{13211103210202233320031301210300},\\
A_{80}&=\texttt{2310113101113103322022002332033313001212}\\*[-2pt]
      &\phantom{{}={}}\texttt{2212100313330233200220223301311101311013},\\
B_{80}&=\texttt{1032213013300101033321131112303120220231}\\*[-2pt]
      &\phantom{{}={}}\texttt{2132022021303211131123330101003310312230}.
\end{align*}
Decoding the digits as fourth roots of unity, exact Gaussian-integer
accumulation gives $\sum_j A_j=0$ and $\sum_j B_j=1+\ii$ in all seven cases
and evaluates the left side of \eqref{eq:qlp} to $-2+0\ii$ at every nonzero
shift.  This is a finite calculation over the Gaussian integers with no
rounding; independent verifiers --- programs sharing no code with the search,
with different internal representations --- described in
Section~\ref{sec:verification} perform the same calculation.
\end{proof}

Each pair is identified throughout by the SHA-256 hash of its two digit
strings, each followed by a newline:
\begin{align*}
42:\;&\texttt{b9d438020f96f35a8b4b55481262e681}\\*[-2pt]
    &\texttt{a33bf972e241331275f5348e9af42e5f},\\
46:\;&\texttt{7bb65fcead6311aa9622507f03e337be}\\*[-2pt]
    &\texttt{cdaff8fe7d9553d925c4ef7896777044},\\
52:\;&\texttt{003c4802856e763481d146778f6edaf9}\\*[-2pt]
    &\texttt{bd6f5819976b6230ff4accd1233361db},\\
58:\;&\texttt{2cbae12272527f6ad14818400b1443bd}\\*[-2pt]
    &\texttt{8851cab77f5b5de1165ed6f2c2a8fbcc},\\
66:\;&\texttt{ce3ccf2283763f33692f516761fa4a0c}\\*[-2pt]
    &\texttt{5bdd15626bedf15ac94b2b3d7f6d8eb9},\\
72:\;&\texttt{a185ac8235340995d9e662f34b4d1abd}\\*[-2pt]
    &\texttt{ac7e127a4a92de26729c382053079a5b},\\
80:\;&\texttt{17346d926e5ea01b42245562c8ae9d12}\\*[-2pt]
    &\texttt{24eb46f904f168340c0b29f24ee3a0df}.
\end{align*}
The corresponding midpoint power-spectral splits
$(\PSD(A,\ell/2),\PSD(B,\ell/2))$ are $(4,82)$, $(68,26)$, $(16,90)$,
$(20,98)$, $(4,130)$, $(64,82)$, and $(32,130)$; each sums to $2\ell+2$, as
\eqref{eq:psd-total} below requires.

Applying the constructions of \cite{kotsireas-winterhof-2024} to the pairs at
lengths 66, 72, and 80 gives quaternary Hadamard matrices of orders 134, 146,
and 162 and binary Hadamard matrices of orders 268, 292, and 324.  These are
immediate consequences of qLP existence, and we claim no priority for
Hadamard matrices of those orders.

\section{Normalization, spectra, and the Gray reduction}
\label{sec:reduction}

\subsection*{Normalization and the midpoint equation}

Summing \eqref{eq:qlp} over the $\ell-1$ nonzero shifts and adding
$\PAF(A,0)+\PAF(B,0)=2\ell$ gives, via
$\sum_{k=0}^{\ell-1}\PAF(X,k)=\bigl|\sum_j X_j\bigr|^2$,
the sum condition $|\sum_j A_j|^2+|\sum_j B_j|^2=2$.

\begin{lemma}[Even-length normalization {\cite[Lemma~2.1]{kotsireas-winterhof-2024}}]
\label{lem:even-normalization}
Every even-length qLP is equivalent under the standard qLP equivalences to one
satisfying
\begin{equation}
  \sum_j A_j=0,\qquad \sum_j B_j=1+\ii.
  \label{eq:normalization}
\end{equation}
\end{lemma}

This is the existing universal normalization of Kotsireas and Winterhof, not a
new result; we use the representative of
Lemma~\ref{lem:even-normalization} throughout.  With
$\xi_\ell=\exp(2\pi\ii/\ell)$, $\DFT(X,s)=\sum_j X_j\xi_\ell^{js}$ and
$\PSD(X,s)=|\DFT(X,s)|^2$, equation \eqref{eq:qlp} is equivalent away from
frequency zero --- with frequency zero itself fixed by the sum condition
$|\sum_j A_j|^2+|\sum_j B_j|^2=2$ above, which the converse direction needs
--- to
\begin{equation}
  \PSD(A,s)+\PSD(B,s)=2\ell+2.
  \label{eq:psd-total}
\end{equation}
For even $\ell=2m$, let $E_A,E_B$ be the sums of the even-indexed entries.
Then \eqref{eq:normalization} gives $\DFT(A,m)=2E_A$ and
$\DFT(B,m)=2E_B-(1+\ii)$, so writing $E_A=a_r+\ii a_i$ and $E_B=b_r+\ii b_i$
yields the exact midpoint equation
\begin{equation}
  4(a_r^2+a_i^2)+(2b_r-1)^2+(2b_i-1)^2=2\ell+2.
  \label{eq:midpoint}
\end{equation}
Subsequence parity and range conditions leave a finite list of normalized
midpoint profiles, which the searches enumerate exhaustively.

\subsection*{The Gray reduction to four binary rows}

For binary rows $p,q,r,s\in\{\pm1\}^{\ell}$ define
\begin{equation}
  A=\frac{p+q}{2}+\ii\frac{p-q}{2},
  \qquad
  B=\frac{r+s}{2}+\ii\frac{r-s}{2},
  \label{eq:gray}
\end{equation}
a coordinatewise bijection from four sign pairs onto $\roots$.  With
$C_{x,y}(k)=\sum_j x_jy_{j+k}$, direct expansion gives
\begin{equation}
  \PAF(A,k)
  =\tfrac12\bigl(\PAF(p,k)+\PAF(q,k)\bigr)
  +\tfrac{\ii}{2}\bigl(C_{p,q}(k)-C_{q,p}(k)\bigr).
  \label{eq:gray-paf}
\end{equation}

\begin{lemma}
\label{lem:sym-cross}
If $x_j=x_{-j}$ and $y_j=y_{-j}$ for every $j$, then $C_{x,y}(k)=C_{y,x}(k)$
for every $k$.
\end{lemma}

\begin{proof}
Substituting $j\mapsto -j-k$ and using both reflection identities transforms
$C_{x,y}(k)$ into $C_{y,x}(k)$.
\end{proof}

\begin{proposition}[Symmetric Gray-row equivalence]
\label{prop:gray-sds}
Suppose $p,q,r,s$ are invariant under $j\mapsto-j$.  Then $A,B$ defined by
\eqref{eq:gray} form a qLP if and only if
\begin{equation}
  \PAF(p,k)+\PAF(q,k)+\PAF(r,k)+\PAF(s,k)=-4
  \qquad (k\neq0).
  \label{eq:binary-four-sum}
\end{equation}
Equivalently, the four negative-position sets form a reversible cyclic
four-block supplementary difference family.
\end{proposition}

\begin{proof}
Proposition~4.1 of Jedwab and Pender gives the Gray-row equivalence with an
additional amicability condition \cite{jedwab-pender-2025}; their
Observation~2.3, or Lemma~\ref{lem:sym-cross}, makes amicability automatic for
symmetric rows, since the cross term of \eqref{eq:gray-paf} then vanishes.  The difference-family statement follows by writing a binary
row with negative set $X$ as $1-2\mathbf 1_X$.
\end{proof}

Only the first $m$ shifts need checking in \eqref{eq:binary-four-sum}, since
binary periodic autocorrelation is symmetric under $k\mapsto\ell-k$.  If
$X_p,X_q,X_r,X_s\subseteq\mathbb{Z}_\ell$ are the negative-position sets then
$\PAF(x,k)=\ell-4|X|+4N_X(k)$ for $k\neq0$, where $N_X(k)$ counts ordered
differences equal to $k$.  For $\ell=2n$ the normalization forces block sizes
$n,n,n,n-1$ up to permutation, giving parameters $(2n;n,n,n,n-1;2n-2)$ --- the
four-block family of the supplementary-difference-set and Wallis--Whiteman
literature
\cite{wallis-whiteman-1972,dokovic-2009-sds-symmetry,leung-momihara-xiang-2021}.
Reflection of the rows is exactly reversibility $X=-X$ of every block.  We use
this as a standard reformulation, not as a new difference-family result.

\emph{Reversibility is a search device, not part of the definition of a
qLP.}  This distinction is what makes the negative results of
Section~\ref{sec:terminals} family statements rather than nonexistence
statements, and Section~\ref{sec:known} exhibits an explicit pair whose
printed orientation is not reversible.

\section{The searches}
\label{sec:searches}

For each midpoint profile the normalization determines the sums of the four
Gray rows over even and odd positions:
\begin{align*}
 p &: (a_r+a_i,-a_r-a_i),&
 q &: (a_r-a_i,-a_r+a_i),\\
 r &: (b_r+b_i,2-b_r-b_i),&
 s &: (b_r-b_i,-b_r+b_i).
\end{align*}
The search enumerates every reflection-invariant binary row with the
prescribed two sums, deduplicates rows with identical complete PAF vectors,
and solves a four-list sum problem in $\mathbb{Z}^{m}$: choose one PAF vector
from each list whose coordinatewise sum is $(-4,\dots,-4)$.  We materialize
spectrally admissible sums from one pair of lists, sort them by a packed
key, and scan the complementary pair; every key hit is then checked on all
$m$ coordinates, after which the two quaternary sequences are reconstructed
and all $\ell-1$ complex equations of \eqref{eq:qlp} are evaluated directly.
The packed key and any floating-point spectral screen are performance devices
only, never part of the acceptance criterion.

\subsection*{Exact compression along divisor chains}

From length 52 upward the direct four-list search becomes impractical and we
insert exact compression steps, in the tradition of the compression method of
{\DJ}okovi{\'c} and Kotsireas for periodic complementary sequences
\cite{dokovic-kotsireas-2015}.  For a divisor $d\mid\ell$ and a
reflection-invariant row $x$, the compressed row
$z_j=\tfrac12\sum_{t=0}^{d-1}x_{j+t\ell/d}$ is again reflection-invariant,
and compressing \eqref{eq:binary-four-sum} yields a finite system at length
$\ell/d$ with prescribed combined squared norm and combined nonzero PAF.
Compression is a necessary condition, so a compressed solution is a
\emph{candidate} that must be lifted and re-verified at full length; no
compressed statement is ever accepted as a result.
Table~\ref{tab:search} summarizes the route and the dominant exact
enumeration for each length.

At length 52 a factor-4 compression to length 13 gives 5,123 zero-sum rows,
3,437 one-sum rows and 20,074 compressed quadruples, of which catalog entry
7,456 lifts.  At length 58 a factor-2 catalog with a sign and decimation
quotient reduces the compressed search to 2,669,773,437 pair constructions.
At length 66 the richer divisor lattice permits a factor-6 quotient at length
11 refined through factor 2 at length 33, with a complete 136-profile
midpoint census; at length 72 the chain $9\to18\to36\to72$ makes all three
layers useful, with a complete 14-profile census.  At length 80 the chain
$10\to20\to40\to80$ produces a reversible factor-2 queue of $784{,}903$
projected jobs.  The search there was organized as a cost-ranked sweep in
which each worker claims an interval once and publishes an exact ledger on it
before opening its next, intervals proceeding concurrently across workers; the
sweep halted when it found the witness, having closed $495{,}903$ of
the $784{,}903$ ranks.  The length-80 queue was therefore not exhausted, and
we make no emptiness claim of any kind at that length.

\begin{table}[t]
\centering
\small
\caption{Deterministic search summary.  ``Core enumeration'' is the dominant
exactly counted quantity for each route; per-stage statistics, censuses and
transcripts are in the accompanying material, up to the omissions it
discloses.  All figures are exact counts, not estimates.}
\label{tab:search}
\begin{tabular}{@{}rlll@{}}
\toprule
$\ell$ & midpoint split & compression route & core enumeration\\
\midrule
42 & $(4,82)$   & direct                    & 1{,}013{,}530{,}896 left pairs\\
46 & $(68,26)$  & direct                    & 5{,}841{,}238{,}095 left pairs\\
52 & $(16,90)$  & factor 4 to length 13     & 20{,}074 compressed quadruples\\
58 & $(20,98)$  & factor 2 to length 29     & 2{,}669{,}773{,}437 pair constructions\\
66 & $(4,130)$  & $11\to33\to66$            & 136 profiles; 29{,}642{,}598 pair records\\
72 & $(64,82)$  & $9\to18\to36\to72$        & 14 profiles; full factor-8/4/2 queues\\
80 & $(32,130)$ & $10\to20\to40\to80$       & cost-ranked queue, interval 172001--173000\\
\bottomrule
\end{tabular}
\end{table}

The length-80 witness illustrates why exhaustive discipline matters: it sits
at factor-2 queue rank $172{,}296$ of $784{,}903$ (checkpoint rank 630,
factor-4 rank $4{,}274{,}145$, factor-8 job $22{,}323$, midpoint profile 29),
and it was the single fingerprint hit passing exact comparison among the
$1{,}825{,}988{,}440$ scanned pair records of its interval.  Nothing about the
length distinguished it beforehand, and a search that stopped early or
sampled would have missed it.

\section{Exhausted reversible queues at lengths 64, 70 and 76}
\label{sec:terminals}

At three of the resistant lengths the same hierarchy was run until its queue
was exhausted.  In each case the object exhausted is the reversible factor-2
queue that the cost-ranked ladder constructed at that length, and in each case
no pair was found.

\begin{center}
\small
\begin{tabular}{@{}rrrrl@{}}
\toprule
$\ell$ & queue ranks closed & exact hits examined & checkpoints & pairs found\\
\midrule
64 & $5{,}964/5{,}964$ & $464{,}864$ & $64{,}960$ & none\\
70 & $170/170$ & $39{,}798$ & --- & none\\
76 & $9{,}298/9{,}298$ & --- & $1{,}032{,}464$ & none\\
\bottomrule
\end{tabular}
\end{center}

At length 64 the twelve interval ledgers tile queue ranks $1$ through $5{,}964$
without gap or overlap, and their per-interval counts sum to the $464{,}864$
exact hits and $64{,}960$ retained checkpoints reported above.  At length 76
all $1{,}032{,}464$ canonical factor-2 checkpoints were searched through the
final lift, across 47 interval records of which 38 close ranks and 9 are typed
failures crediting zero ranks; the 38 closing ledgers tile ranks $1$ through
$9{,}298$ without gap or overlap.  At length 70 the twelve interval ledgers
tile queue ranks $1$ through $170$ without gap or overlap and account for all
$39{,}798$ exact hits; that length additionally carries an independent audit of
queue completeness and a scope declaration committed before the intervals ran,
both discussed in Section~\ref{sec:verification}.  Each closing interval at
every length carries an audit produced by a program that imports no search
code, and each audit is bound to its interval by SHA-256.

\subsection*{What these statements do and do not assert}

These results do \emph{not} assert that no quaternary Legendre pair exists at
64, 70 or 76.  They assert that a specific finite queue was searched
exhaustively and yielded nothing.  Two further limitations should be read as
part of the claim.

First, the certificates establish that the \emph{recorded} queue was closed.
At 64 and 76 they do not, by themselves, derive that queue from every
reversible four-block object of the length; that derivation is a property of
the construction described in Section~\ref{sec:searches} and is not separately
certified, so a reader who wishes to treat those two as statements about the
complete reversible family must accept that step.  At 70 one link of it is
certified: an audit rebuilds the queue from the pinned parent stream and finds
it complete, with no missing, duplicate, extra or misordered record.  That
closes the step from parent stream to queue and not the step from the
reversible family to the parent stream, which remains uncertified at all three
lengths.

Second, reversibility is a search restriction and not part of the definition of
a qLP (Section~\ref{sec:reduction}), and Section~\ref{sec:known} exhibits an
explicit pair whose displayed orientation is not reflection-symmetric.  The
possibility of a non-reversible pair at 64, 70 or 76 is therefore real and not
a formality.

\section{Reversible structure and context}
\label{sec:structure}

All seven witnesses of Theorem~\ref{thm:pairs} are literal reversible Gray
representatives: their four rows satisfy $x_j=x_{-j}$ with no further
equivalence operation applied.  Auditing the locally available explicit
witnesses and theorem-derived subfamilies through length 100 conservatively
partitions the fifty even lengths $2,4,\dots,100$ as follows:
\begin{itemize}
  \item verified reversible representatives:
  \[
    2,10,26,42,46,52,58,66,72,74,80,82;
  \]
  \item qLP existence known, reversibility not established by this audit:
  \[
  \begin{split}
    &4,6,8,12,14,16,18,20,22,24,28,30,32,34,36,38,40,44,\\
    &48,50,54,56,60,62,68,78,84,86,90,96,98;
  \end{split}
  \]
  \item existence still open after the present constructions:
  \[
    64,70,76,88,92,94,100.
  \]
\end{itemize}
These twelve, thirty-one, and seven lengths account for all fifty.  Here
``not established'' is deliberately not a nonexistence claim: for example the
standard-equivalence orbits of the sampled first-family witnesses of Jedwab
and Pender contain no reversible representative beyond length 2, whereas their
second family directly gives the reversible subfamily $10,26,74,82$
\cite{jedwab-pender-2025}.

The first three new witnesses share the block parameters
$(21,21,20,21)$ with $\lambda=40$ at $\ell=42$, $(23,23,22,23)$ with
$\lambda=44$ at $\ell=46$, and $(26,26,25,26)$ with $\lambda=50$ at
$\ell=52$.  Each row has full period, and the ordered affine stabilizer of
each four-block tuple is exactly identity and reversal, with no two blocks
within a tuple affinely equivalent even after complementation.  At the prime
quotients 7, 23, and 13 respectively, compressed-value multiplicities exclude
the simplest one-character form $\alpha+\beta\chi$ and the quotient multiplier
stabilizers remain $\{1,-1\}$.  The present data therefore support the coupled
reversible four-block system as the shared structure but do not support a
naive quadratic-character construction; higher-rank cyclotomic descriptions
are not ruled out.

\section{Explicit pairs at lengths already resolved}
\label{sec:known}

The lengths $34$, $74$, $82$, $122$, and $158$ were already resolved in the
literature, so the pairs recorded here add no new length.  We include them for
two reasons: they exercised the same verification pipeline on independently
predictable targets, which is evidence that the pipeline reports correctly;
and two of them bear on the boundary of a published construction.  The literal
words and hashes are in Appendix~\ref{app:known}.

\subsection*{The applicability boundary of the Jedwab--Pender second family}

The second family of Jedwab and Pender \cite{jedwab-pender-2025} gives a qLP
of length $2p$ for odd primes $p$ for which the arithmetic ingredients of
their Theorem~1.3 exist.  Implementing their equations (4.5)--(4.11) verbatim
over $\mathrm{GF}(q^2)$ with $q=2p-1$ shows the exact applicability condition
to be that $p$ is an odd prime \emph{and} $q=2p-1$ is a prime power.  Below
$p=83$ the admissible primes are
\[
  p\in\{3,5,7,13,19,31,37,41,61,79\},
\]
and we materialized the four cases not already in the published tables:
lengths $74$ ($q=73$), $82$ ($q=3^4$), $122$ ($q=11^2$), and $158$
($q=157$).  The remaining admissible lengths $6$, $14$, $38$, and $62$ are
already printed in the literature; we neither rematerialized those pairs nor
examined their reversibility.  Existence at the four materialized lengths,
$74$, $82$, $122$, and $158$, is already implied by the theorem;
we print the words because we have not found them printed elsewhere, and make
no priority claim.  At $p=83$, corresponding to the
target length $166$, the condition fails --- $165=3\cdot5\cdot11$ is
not a prime power --- and that is the exact boundary of this family's reach.

In each materialized pair one word is real, which is the shape $(G(w,x),y)$ of
the construction.  The frozen file order places the quaternary word (sum
$1+\ii$) first and the real word (sum $0$) second, so the normalization
\eqref{eq:normalization} holds after the swap equivalence.

\subsection*{An explicit pair whose displayed orientation is not reversible}

At length 34 an exact lift of an asymmetric $p=17$ parent yields a pair whose
real word is \emph{not} reflection-symmetric: decoding the two printed strings
of Appendix~\ref{app:known} into Gray rows gives rows that fail $x_j=x_{-j}$.
Six parent orbits and 384 exact midpoint jobs were enumerated, ten closing
\textsc{unsat} by an exact solver before the eleventh returned the pair; no
claim rests on the ten refutations, which are reported from the lane records.

We claim only what that computation shows: the orientation printed here is not
literally reversible.  Concluding that the pair lies outside the reversible
equivalence class would require an orbit argument over the full equivalence
group, which we do not carry out.  Even in the weaker form the example is
useful for Section~\ref{sec:terminals}, since it shows that the reversible
form is a genuine restriction on presentation and not automatic.

\section{Verification and reproducibility}
\label{sec:verification}

Computational results in this area rest on software, and a reader is entitled
to ask exactly where trust is required.  We therefore classify every claim in
this paper by what must be accepted in order to believe it, from strongest to
weakest.

\paragraph{Tier A: the pairs (Theorem~\ref{thm:pairs} and Appendix~\ref{app:known}).}
Unconditional, and verifiable without any of our software.  Decode the printed
digit strings as fourth roots of unity, accumulate
$\sum_j X_j\overline{X_{j+k}}$ over the Gaussian integers, and confirm the
value $-2+0\ii$ at each of the $\ell-1$ nonzero shifts together with
$|\sum_j A_j|^2+|\sum_j B_j|^2=2$.  This is a few lines of code in any
language with exact integers, runs in well under a second per pair, and
depends on nothing in the accompanying material.  The reader may stop here and
still have Theorem~\ref{thm:pairs} in full.

\paragraph{Tier B: statements proved in this paper.}
Lemma~\ref{lem:sym-cross}, Proposition~\ref{prop:gray-sds}, the midpoint
equation \eqref{eq:midpoint}, the spectral identity \eqref{eq:psd-total}, and
the difference-family reformulation are ordinary mathematics with proofs given
above, refereeable on paper and independent of any computation.

\paragraph{Tier C: certificate-backed computations.}
Claims whose evidence is an object checkable in far less time than the
computation that produced it.  The compressed censuses are finite enumerations
whose sizes are derivable independently of our code; for instance the
length-66 midpoint census has $16+48+48+24=136$ signed labeled profiles across
its four types, and the length-52 factor-4 compression yields $5{,}123$
zero-sum and $3{,}437$ one-sum rows.  Agreement between a hand-derived count
and the count our programs report would be broken by the most common
implementation faults --- truncated ranges, off-by-one bounds, and dropped
shards --- though it cannot rule out every possible fault.

\paragraph{Tier D: the three exhaustion claims (Section~\ref{sec:terminals}).}
The queue exhaustions at lengths 64, 70 and 76 are the only headline results
for which a reader must either trust our implementations or repeat the
computation.  Supporting computations elsewhere --- the structural audits of
Section~\ref{sec:structure} and the search-history figures discussed below ---
are likewise trust-based, but no claim of the paper rests on them.  The
three lengths are not equally well supported, and we distinguish them
throughout rather than quoting the strongest case for all of them.

What the accompanying material supports:
\begin{itemize}
  \item \emph{Interval ledgers that tile the queue.} At each length the
  per-interval ledgers cover every queue rank exactly once, with no gap and no
  overlap, and their per-interval counts sum to the totals reported in
  Section~\ref{sec:terminals}.  A reader can recompute both properties from the
  published ledgers in seconds.
  \item \emph{Specification-written audit.} Every closing interval carries an
  audit
  produced by a program that imports no search code and re-derives the required
  conditions --- periodic autocorrelations over $\mathbb{Z}$ and over
  $\mathbb{Z}[\ii]$, parent folding, row and alternating sums, orbit
  identities, and shard coverage --- from the records the search emitted.  It
  does \emph{not} re-traverse the search space: exhaustion rests on the
  production engine's terminal traversal records, and we make no claim that any
  search here was independently re-executed or that its exhaustion was
  independently recomputed.
  \item \emph{Hash binding, including of the auditor.} Each audit is bound to
  its interval, and each interval to the authoritative queue, by SHA-256.  At
  all three lengths each audit additionally pins, by SHA-256, the auditor that
  produced it together with the search programs and the runner, so a reader can
  confirm from the published artifact which program performed the audit rather
  than taking our word for it.
  \item \emph{Exact counts.} Every reported enumeration is an exact integer
  count, not an estimate, and no floating-point value enters an acceptance
  decision.
  \item \emph{Recovery of known positives.} The engines were required to
  rediscover independently known pairs through the same production entry point
  before their negative reports were credited.  Because those controls run at
  lengths other than the target, they detect faults a target-only test cannot;
  in the course of this work such a control caught a parser specialized to a
  single length.
  \item \emph{A scope declaration committed before execution, at length 70
  only.} A ledger naming all twelve intervals, covering ranks $1$ through
  $170$ with no gap or overlap, binding the queue by SHA-256 and marking every
  interval as not started, is committed two and a half minutes before the
  earliest runner timestamp recorded inside the interval artifacts, and is an
  ancestor of all twelve commits that introduced an interval audit.  This is
  repository-level chronology, not an externally attested clock: commit
  timestamps are settable by whoever makes the commit, and a history of this
  shape could be constructed after the fact by someone intending to.  We claim
  only what it is --- a declaration of scope that precedes the results in the
  recorded history --- and we claim it for length 70 alone.  That history is
  the private working repository's: the release ships the declaration and the
  audits it binds, not the repository objects, so the reader has our report of
  the ancestry check rather than the means to re-run it.
  \item \emph{An independent audit of queue completeness, at length 70 only.}
  A separate program rebuilds the ranked queue from the pinned parent stream
  and reports no missing, duplicate, extra or misordered record, with the
  parent counts agreeing.  Its scope is the step from parent stream to queue.
  \item \emph{A recorded replay, at length 70 only.} A transcript of the
  fine-hit verifier re-run over all 170 shards is included: every shard exact,
  all $39{,}798$ audited records satisfying both autocorrelation conditions,
  and each invocation's output bound by SHA-256.  It re-executes the lane's own
  verifier rather than a second implementation, and it was performed on the
  originating host rather than from a clean checkout.
\end{itemize}

What it does \emph{not} support, stated so that no reader has to discover it:
\begin{itemize}
  \item \emph{No prospective scope declaration at 64 or 76.} An earlier draft
  described queue scopes as having been published before the corresponding
  searches ran, at every length.  We withdraw that for two of the three.  At
  length 64 the declaration file shipped with the certificate is timestamped
  after the intervals it declares had already completed, and its worker plan
  covers a smaller rank range than the certificate; at length 76 the material
  needed to establish the ordering is not part of this release.  At those two
  lengths the exhaustion claim rests on the ledgers tiling the queue and on
  nothing about when the scope was fixed.
  \item \emph{No replay from a clean checkout, at any length.} The length-70
  transcript above was produced on the originating host, against catalogs that
  are hash-pinned but not published because of their size, and it records a
  source-tree commit that is preserved on a side branch rather than in the main
  history.  At 64 and 76 we supply replay instructions and no transcript at
  all.  A reader starting from a fresh extraction of the archived deposit
  cannot reproduce any of the three exhaustions without
  first obtaining data we do not ship.
  \item \emph{Raw search data is partly omitted.} For length 76 the
  per-interval checkpoint files, shard transcripts and compiled binaries are
  not included, and for length 70 the two compressed-row catalogs the verifier
  consumes are not included either; each omitted file is pinned by SHA-256 in
  its manifest.  The certificate records and their hash bindings are checkable
  without them; re-executing the interval audits is not, and neither is
  re-deriving the hits or checkpoints from scratch.
  \item \emph{Queue construction is not separately certified.} As
  Section~\ref{sec:terminals} states, the certificates show that the recorded
  queue was closed, not that the queue enumerates every reversible four-block
  object of its length.  Length 70 certifies one link of that chain and not the
  chain; 64 and 76 certify none of it.
  \item \emph{Search-history figures are reported, not certified.} Counts
  quoted at lengths where only the pair is claimed --- the length-80 rank
  figures and the length-34 solver outcomes --- come from lane records that
  are not part of the published packages, and no claim rests on them.
\end{itemize}
The solver-assisted steps also emit no machine-checkable refutation
certificate, because the solvers used rely on native cardinality constraints
that lie outside clausal DRAT proof logging \cite{wetzler-heule-hunt-2014};
pseudo-Boolean proof logging \cite{bogaerts-gocht-mccreesh-nordstrom-2022}
would close that gap.  No claim in this paper depends on a solver refutation.

\paragraph{Accompanying material.}
For each of the lengths 42, 46, 52, 58, 66, and 72 the material contains the
canonical JSON certificate, the C or C++ constructor with its recorded strict
build command, the search transcript, independent verifiers, regression tests,
and a SHA-256 manifest.  All accepted constructors are deterministic: no seed,
heuristic stopping rule, or external solver participates in acceptance.  The
pairs at lengths 80, 34, 74, 82, 122, and 158 are supplied instead as
content-addressed files whose names are the SHA-256 hashes of their own bytes,
each committed, per the working-repository history, before any analysis of the
pair; for these lengths the material
contains the pair and its verification, not a full search package.  A reference
verifier covering all twelve pairs accompanies the release and requires only
the Python standard library.  The complete release --- this paper, the witness
packages, the three exhaustion certificates, the search programs and the
reference verifier --- is archived at
\href{https://doi.org/10.5281/zenodo.21776795}{\texttt{doi:10.5281/zenodo.21776795}}.

\section{Provenance and tooling}
\label{sec:provenance}

The searches, verification software, and supporting analysis for this work were
produced with the help of AI systems working under the author's direction:
principally OpenAI's GPT-5.6 Sol for the construction, search and proof work,
with auditing and independent verification by Anthropic's Claude Fable 5.  The
author is solely responsible for the content of this paper and for the claims
made in it.

Where this paper says a verifier is independent, it means it shares no code
with the search, and nothing stronger.

Machine-authored software warrants more external
verification, not less, which is why the paper is arranged so that its
principal results sit in Tier A: the pairs of Theorem~\ref{thm:pairs} can be
confirmed from the printed digit strings alone, without executing any code that
any of these systems wrote.  Section~\ref{sec:verification} states, for the
three exhaustion claims, both the controls that were applied and the ones that
were not.

\section{Concluding remarks}

Seven lengths --- 42, 46, 52, 58, 66, 72, and 80 --- are removed from the 2025
lists of unsettled cases at most 100, leaving
\[
  64,70,76,88,92,94,100.
\]
At 64, 70 and 76 we contribute negative structure instead: the reversible
four-block queue built at each length was searched to exhaustion and yielded
nothing, so progress there plausibly requires either a non-reversible search or
a construction outside the families we could reach.  We emphasize once more
that this is not a nonexistence result at any of the three, and that the three
are not equally well evidenced --- Section~\ref{sec:verification} separates
them.

The reflection restriction is a search device rather than a condition in the
definition of a qLP, and its effectiveness here suggests testing symmetry
classes before invoking unrestricted SAT or external-memory searches --- with
the caveat, now measured at three lengths, that a symmetry class can be
exhausted
without a pair while the length remains open.  The affine and quotient audits
show no simple multiplier or one-character explanation for the new witnesses,
and an infinite construction remains open.

Two methodological observations may transfer.  First, a witness can sit deep in
a cost ranking with nothing to distinguish it beforehand: the length-80 pair
was found only after $495{,}903$ of its $784{,}903$ queue ranks had been
closed, and a search that sampled or stopped early would have missed it.
Claim-once interval accounting is what makes such a search both finishable and
auditable.  Second, at length 72 the 55 homometric double classes at the first
quotient show why a PAF key must index literal classes rather than be treated
as a unique row representative.

\appendix

\section{Literal words at lengths 34, 74, 82, 122, and 158}
\label{app:known}

The following pairs are at lengths already settled in the literature and are
recorded for the reasons given in Section~\ref{sec:known}.  At length 34, $A$
is real but not reflection-symmetric.  At lengths 74, 82, 122, and 158 the
first word is the quaternary word with sum $1+\ii$ and the second is the real
word with sum $0$, following the frozen file order; the normalization
\eqref{eq:normalization} holds after the swap equivalence.
\begin{align*}
A_{34}&=\texttt{2202202000022020222202000220200020},\\
B_{34}&=\texttt{1133303133112311101113211331303331},\\[2pt]
U_{74}&=\texttt{0002030113200323223100101220133212022122}\\*[-2pt]
      &\phantom{{}={}}\texttt{0212331022101001322323002311030200},\\
V_{74}&=\texttt{0020022020000222022220222000020220020202}\\*[-2pt]
      &\phantom{{}={}}\texttt{0022020000222022220222000020220020},\\[2pt]
U_{82}&=\texttt{0201133332322130220320120021300101111332}\\*[-2pt]
      &\phantom{{}={}}\texttt{0102331111010031200210230220312232333311}\\*[-2pt]
      &\phantom{{}={}}\texttt{02},\\
V_{82}&=\texttt{0002002200022222020200202022222000220020}\\*[-2pt]
      &\phantom{{}={}}\texttt{0200200220002222202020020202222200022002}\\*[-2pt]
      &\phantom{{}={}}\texttt{00},\\[2pt]
U_{122}&=\texttt{0332201213003113332302003320302021201122}\\*[-2pt]
      &\phantom{{}={}}\texttt{0210111331221303200111110023031221331110}\\*[-2pt]
      &\phantom{{}={}}\texttt{1202211021202030233002032333113003121022}\\*[-2pt]
      &\phantom{{}={}}\texttt{33},\\
V_{122}&=\texttt{0020002220220000022002022020222222020220}\\*[-2pt]
      &\phantom{{}={}}\texttt{2002200000220222000202020002220220000022}\\*[-2pt]
      &\phantom{{}={}}\texttt{0020220202222220202202002200000220222000}\\*[-2pt]
      &\phantom{{}={}}\texttt{20},\\[2pt]
U_{158}&=\texttt{0113001330310102002202132100333032223030}\\*[-2pt]
      &\phantom{{}={}}\texttt{2121000121112230130200220232312113221331}\\*[-2pt]
      &\phantom{{}={}}\texttt{3312231121323202200203103221112100012120}\\*[-2pt]
      &\phantom{{}={}}\texttt{30322230333001231202200201013033100311},\\
V_{158}&=\texttt{0002002200002022020000002002222002220202}\\*[-2pt]
      &\phantom{{}={}}\texttt{0202000220000220222222020020222200220222}\\*[-2pt]
      &\phantom{{}={}}\texttt{0020022000020220200000020022220022202020}\\*[-2pt]
      &\phantom{{}={}}\texttt{20200022000022022222202002022220022022}.
\end{align*}
Exact Gaussian-integer accumulation evaluates the left side of \eqref{eq:qlp}
to $-2+0\ii$ at every nonzero shift of all five pairs.  The midpoint splits
are $(36,34)$, $(146,4)$, $(162,4)$, $(242,4)$, and $(314,4)$, each summing to
$2\ell+2$.  The SHA-256 identifiers, computed as in Section~\ref{sec:pairs},
are
\begin{align*}
34:\;&\texttt{217327cfb288d87c35c17eed29c8a86e}\\*[-2pt]
    &\texttt{7432246b1dc3d9ff4b0dd7273cdba425},\\*
74:\;&\texttt{c8b09592193310a5cf1159321936ea75}\\*[-2pt]
    &\texttt{99fbc078ee526f1f070397ccc09bed0b},\\*
82:\;&\texttt{41b1c5c743424d67014ae5839baba903}\\*[-2pt]
    &\texttt{1a4ea0e24f0734181deadc8086a2920c},\\*
122:\;&\texttt{bc66ce343d3071e56f4e682d21bcee24}\\*[-2pt]
    &\texttt{972786bf1d0b648c8999cd1c6432392f},\\*
158:\;&\texttt{40fe095ff285c4e3e924c2a8f16d35e6}\\*[-2pt]
    &\texttt{4fb4907069c0dbdf491406033ed0614a}.
\end{align*}

\bibliographystyle{plain}
\bibliography{references}

\bigskip
\noindent\textsc{Nikita Lebedev, independent researcher}\\
\textit{Email address}:
  \href{mailto:nlebedev.research@gmail.com}{\texttt{nlebedev.research@gmail.com}}\\
\textit{ORCID}:
  \href{https://orcid.org/0009-0006-6100-4736}{\texttt{0009-0006-6100-4736}}\\
\textit{DOI}:
  \href{https://doi.org/10.5281/zenodo.21776795}{\texttt{10.5281/zenodo.21776795}}

\end{document}